\documentclass[12pt]{amsart}

\usepackage{amsmath, amsthm, amscd, amsfonts, amssymb, graphicx, color}
\usepackage{latexsym}
\usepackage{amscd}
\usepackage{amsthm}
\usepackage{amssymb}
\usepackage{enumerate}
\usepackage{mathabx}

\theoremstyle{plain}
\newtheorem{theorem}{Theorem}
\newtheorem{corollary}[theorem]{Corollary}

\newtheorem{prop}[theorem]{Proposition}
\theoremstyle{remark}
\theoremstyle{definition}

\begin{document}

\title{On certain generalizations of the Levi-Civita and Wilson functional equations}
\author{J.~M.~Almira$^*$ and E. V. Shulman}


\thanks{$^*$ Corresponding author}

\address{Departamento de Matem\'{a}ticas, Universidad de Ja\'{e}n, E.P.S. Linares,  Campus Cient\'{\i}fico Tecnol\'{o}gico de Linares, 23700 Linares, Spain}
\email{jmalmira@ujaen.es}
\address{Department of Mathematics, Vologda State University, S. Orlova~ 6, Vologda 160035, Russia;\newline 
and \newline
Instytut of Mathematics, University of Silesia, Bankowa 14,  PL-40-007 Katowice, Poland}
\email{shulmanka@gmail.com, ekaterina.shulman@us.edu.pl}

\begin{abstract}
We study the  functional equation
\[
\sum_{i=1}^mf_i(b_ix+c_iy)= \sum_{k=1}^nu_k(y)v_k(x)
\]
with $x,y\in\mathbb{R}^d$ and $b_i,c_i\in {GL}(d,\mathbb{R})$, both in the classical context 
of continuous complex-valued functions and in the  framework of complex-valued Schwartz 
distributions, where these equations are properly introduced in two different ways. 
The solution sets are, typically,  exponential polynomials and, in some particular cases, related 
to so called characterization problem of the normal distribution 
in Probability Theory, they reduce to ordinary polynomials.
\end{abstract}

\maketitle

\markboth{J.~M.~Almira, E. Shulman}{On certain generalizations of the Levi-Civita and Wilson functional equations}

\section{Introduction}
Levi-Civita functional equation has the form
\begin{equation} \label{LC}
f(x+y)=\sum_{k=1}^nu_k(y)v_k(x),
\end{equation}
where $f, u_k, v_k$, ($1\le k\le n$), are complex-valued functions defined on a semigroup $(G,+)$. This equation can be restated by claiming that
$\tau_y(f)\in W$ for all $y\in G$, where $W=\text{span}\{v_k\}_{k=1}^n$ is a finite-dimensional space of functions defined on $G$ and $\tau_y(f)(x)=f(x+y)$.

In this paper we deal with the case $G = \mathbb{R}^d$ and study a more general functional equation
\begin{equation} \label{LCGG}
\sum_{i=1}^m f_i(b_ix+c_iy)= \sum_{k=1}^nu_k(y)v_k(x)
\end{equation}
for all $x,y\in\mathbb{R}^d$, where $f_i, v_k, u_k$, for $1\le i\le m$, $1\le k\le n$,   are  functions defined on $\mathbb{R}^d$, and $b_i,c_i\in {GL}(d,\mathbb{R})$. Our main result is that all continuous solutions of (\ref{LCGG}) are exponential polynomials. Moreover, using a result from \cite{S} we extend this statement to equations of the form
\begin{equation}\label{ras}
\sum_{i=1}^mf_i(b_ix+c_iy)= \sum_{k=1}^nu_k(y)v_k(x) + \sum_{s=1}^N P_s(x)w_s(y)\exp{\langle x,\varphi_s(y)\rangle},
\end{equation}
where $P_s$ are polynomials and functions $w_s, \varphi_s$ are arbitrary.

In one-dimensional case, addition theorems with such a left-hand side were studied  by Wilson \cite{Wil} 
a hundred years ago. Applying an elimination method, which now is a classical one, he showed 
that all continuous solutions of the equation
\begin{equation*} 
\sum_{i=1}^m f_i(\alpha_i x+ \beta_i y)= f(x) + g(y)
\end{equation*}
are polynomials of  degree  not  greater  than $m$.

The equation (\ref{LCGG}) includes the equations in iterated differences:
\begin{equation*}
\sum_{j=1}^m\lambda_j\Delta_y^{j}f=\sum_{k=1}^nu_k(y)v_k \text{ , } y\in\mathbb{R}^d,
\end{equation*}
where $\Delta_y$ is the difference operator $\Delta_y(f)(x) = f(x+y) - f(x)$. Indeed $$\Delta_y^{m}f(x)=1\cdot f(x)+\sum_{p=1}^m\binom{m}{p}(-1)^{m-p}f(I_dx+(pI_d)y),$$
where $I_d$ denotes the identity matrix of size $d$. A simplest example is the Frechet's equation $\Delta_y^{m}f = 0$.

The equation (\ref{LCGG}) extends also the functional equation
 \begin{equation*}
\frac{1}{N}\sum_{k=0}^{N-1} f(z+w^kh)=0 \text{, for all } z,h\in\mathbb{C},
\end{equation*}
where $w$ is any primitive $N$-th root of $1$. This equation, characterizing harmonic polynomials, was introduced by S. Kakutani and M. Nagumo \cite{kn}, Walsh \cite{wal}  in 1930's, and intensively studied by S. Haruki \cite{h1,h2,h3} in  1970's  and   1980's.

Another special case of
(\ref{LCGG}) is the equation
\begin{equation}\label{Ski-ordinary}
\sum_{i=1}^{m}f_i(b_ix+c_iy)= \sum_{i=1}^m f_i(b_ix) +   \sum_{i=1}^m f_i(c_iy),
\end{equation}
which is a linearized (by taking logarithms) form of the
Skitovich-Darmois functional equation:
\begin{equation*}
\prod_{i=1}^m \widehat{\mu_i}(b_ix+c_iy) = \prod_{i=1}^m \widehat{\mu_i}(b_ix)  \prod_{i=1}^m  \widehat{\mu_i}(c_iy).
\end{equation*}
Here, $ \widehat{\mu_i}$ represents the characteristic function of a probability distribution $\mu_i$. This equation is connected to the characterization  problem of normal distributions. Concretely, its study leads to a proof of the following result [Linnik \cite{Linn}, Ghurye-Olkin \cite{Gu_O}, \cite{KLR}]:

 {\it Assume that $X_i$, $i=1,\cdots,m$ are independent $d$-dimensional random vectors such that the linear forms $L_1= b_1^tX_1+...+b_m^tX_m$ and $L_2= c_1^tX_1+...+c_m^tX_m$  are independent, with $b_i, c_i\in {GL}(d,\mathbb{R})$ for $i=1,\cdots,m$. Then $X_i$ is Gaussian for all $i$}.

The equation \eqref{Ski-ordinary} (with the change of matrices $b_i, c_i$ by automorphisms) and its applications to  probability distributions, has been studied in  great detail by Feld'man \cite{Feld}, for functions defined on locally compact commutative groups.

A more general specialization of \eqref{LCGG} was considered by Ghurye and Olkin in \cite{Gu_O}:
\begin{equation}\label{G-O}
\sum_{i=1}^mf_i(x+c_iy)=A(x,y)+B(y,x),
\end{equation}
where $f_i$ map $\mathbb{R}^d$ to $\mathbb{C}$, and the functions $A,B$ are such that, for each $y\in\mathbb{R}^d$, $A(x,y)$ and $B(x,y)$  are polynomials in the variable $x$ with degrees not greater than  $r$ and $s$, respectively (here $r,s$ do not depend on $y$).  This equation has also proven to be a useful tool in the  study of probability distributions (see, for example, \cite[Chapter 7]{MaPe}).

The equation \eqref{LC} can be formulated also for distributions, since the shift operator
$\tau_y: f(x)\mapsto f(x+y)$ and dilation operator $\sigma_b: f(x)\mapsto f(bx)$ can be extended to the space $\mathcal{D}(\mathbb{R}^d)'$  of Schwartz complex-valued distributions (as the adjoint of the corresponding operators on $\mathcal{D}(\mathbb{R}^d)$). Our results in this setting  extend the  Anselone-Korevaar \cite{anselone} theorem on finite-dimensional shift-invariant subspaces of $\mathcal{D}(\mathbb{R}^d)'$ and the results of \cite{leland}, \cite{Lo} (see also \cite{AA_AM} -\cite{AS_DM}). They show in particular that the continuity restrictions on solutions and coefficients of \eqref{LC} can be weakened at least to local integrability.

It should be underlined that in all settings if the functions or distributions $u_k,v_k$ in (\ref{LCGG}) are linearly independent (which always can be assumed), then they are linear combinations of (shifted) functions $f_i$. Therefore, proving that $f_i$ are exponential polynomials we simultaneously prove the same for $u_k$ and $v_k$.

\section{Solution of equation (\ref{LCGG})}

Our aim here is to establish the following result:
\begin{theorem}\label{main}
If the functions $v_k$ are continuous, the matrices $b_i$, $c_i$ and $b_i^{-1}c_i - b_j^{-1}c_j$ (for $i\neq j$) are invertible, then all solutions $f_i\in C(\mathbb{R}^d)$ of \emph{(\ref{LCGG})} are exponential polynomials.
\end{theorem}

Note that the substitution $\widetilde{f}_i(x) = f_i(b_ix)$ reduces the equation (\ref{LCGG}) to the case that $b_i = I_d$, the identity matrix, that is to the equation
\begin{equation}\label{L-Ch-S-R}
\sum_{i=0}^mf_i(x+c_iy)= \sum_{k=1}^nu_k(y)v_k(x)
\end{equation}

 So in what follows we mostly consider this case.

For $y\in \mathbb{R}^d$, let $\tau_y$ denote the shift operator on $C(\mathbb{R}^d)$:
$\tau_y(f)(x) = f(x+y)$. Then denoting by $W$ the subspace generated by $v_1,...,v_n$, we may reformulate (\ref{L-Ch-S-R}) saying that all functions $\sum_{i=1}^m\tau_{c_iy}(f_i)$ belong to $W$.

In the proof of Theorem \ref{main} we will use  the following result from \cite{S}:

\begin{prop} \label{Ekaterina}
Let $\pi$ be a continuous representation of a topologically finitely generated semigroup $G$ on a topological linear space $X$. Suppose that a vector $x\in X$ and a finite-dimensional subspace $L\subset X$  have the property that for any $g\in G$, there is a finite-dimensional $\pi$-invariant subspace $R(g)\subset X$  with
$\pi(g)(x) \in L + R(g)$. Then $x$ belongs to a finite-dimensional $\pi$-invariant subspace of $X$.
\end{prop}

In this section Proposition \ref{Ekaterina} will be applied to the representation $y\mapsto \tau_y$ of the group $\mathbb{R}^d$ by shifts on the space $C(\mathbb{R}^d)$.

By the above, Theorem \ref{main} is equivalent to the following:

\begin{theorem} \label{principalF}
Assume that $\{f_i\}_{i=1}^{m}\subset C(\mathbb{R}^d)$ and, for all $y\in \mathbb{R}^d$,
\begin{equation}\label{inv}
\sum_{i=1}^m\tau_{c_iy}(f_i) \in W \text{ for all } y\in\mathbb{R}^d
\end{equation}
where $W\subset C(\mathbb{R}^d)$ is a finite-dimensional subspace. If all matrices $c_i$ and $c_i-c_j$ (for $i\neq j$) are invertible,
then all $f_i$ are exponential polynomials.
\end{theorem}

\begin{proof}
We use induction on  $m$. The case $m=1$ is known. Indeed, in this case (\ref{L-Ch-S-R}) is the Levi-Civita equation, its solutions on abelian groups are described, for example, in \cite[Theorem 1]{E} (see also \cite{Sz-91, St} and references there). 
 
Suppose that the statement is true whenever the number of summands  appearing in the left-hand side  of the equation is strictly smaller than $m$.
Take arbitrary $h\in \mathbb{R}^d$ and substitute in (\ref{inv})  $y-c_1^{-1}h$ for $y$. Then applying the operator $\tau_h$ we will obtain
\begin{equation}\label{inv1}
\sum_{i=1}^m\tau_{h+c_i(y-c_1^{-1}h)}(f_i) \in \tau_h(W) \text{ for all } y\in\mathbb{R}^d.
\end{equation}
Comparing the left-hand sides of (\ref{inv}) and (\ref{inv1}), and setting  $W^*:=\tau_h(W)+W$ we obtain
\begin{equation}\label{inv2}
\sum_{i=2}^m(\tau_{c_iy +(I_d-c_ic_1^{-1})h}(f_i) - \tau_{c_iy}(f_i))  \in W^*,  \text{ for all } y\in\mathbb{R}^d.
\end{equation}
 Clearly $\dim W^* < \infty$. Setting $d_i = I_d-c_ic_1^{-1}$, define, for a fixed $h$, the functions $g_i$ by $g_i(x) = f_i(x+d_ih) - f_i(x)$. Then
  (\ref{inv2}) will get the form
  \begin{equation}\label{inv3}
\sum_{i=2}^m\tau_{c_iy}(g_i) \in W^* \text{ for all } y\in\mathbb{R}^d.
\end{equation}
By the induction hypothesis, we obtain that all functions $g_i$ are continuous exponential polynomials.

The matrices  $d_i = I_d-c_ic_1^{-1}$ are invertible, since $$\ker d_i = c_1\ker (c_i - c_1) = \{0\}.$$  Thus, the condition ``$f_i(x+d_ih) - f_i(x)$ is a continuous exponential polynomial for all $h$'' can  be written as ``$f_i(x+y) - f_i(x)$ is a continuous exponential polynomial for all $y$''. Since any exponential polynomial is contained in an invariant finite-dimensional subspace, we see that each function $\tau_yf_i$ belongs to the sum of the one-dimensional subspace $\mathbb{C}f_i$ and some invariant finite-dimensional subspace. By Proposition \ref{Ekaterina}, $f_i$ is contained in an invariant finite-dimensional subspace, so $f_i$ is an exponential polynomial. Here $i =2,...,m$, but clearly the same is true for $f_1$ by symmetry.
\end{proof}

\section{ A more general class of equations}

Here we consider the equation (\ref{ras}). Since the second term in the right-hand side of (\ref{ras}) is an exponential polynomial in $x$ for each $y$, the study of this equation reduces to the  following extension of Theorem \ref{principalF}.

\begin{theorem} \label{principalR}
Let $\{f_k\}_{k=1}^{m}\subset C(\mathbb{R}^d)$, $W$ be a finite-dimensional subspace of $C(\mathbb{R}^d)$, $c_i \in GL(d,\mathbb{R})$. Suppose that, for each $y\in \mathbb{R}^d$, there is a finite-dimensional translation invariant space $R(y)\subset C(\mathbb{R}^d)$ with
\begin{equation}\label{fund_gen}
\sum_{i=1}^m \tau_{c_iy}(f_i) \in W +R(y).
\end{equation}
If all matrices $c_i-c_j$ (for $i\neq j$) are invertible, then all $f_i$ are exponential polynomials.
\end{theorem}

\begin{proof} We proceed by induction on $m$. The case $m=1$ is solved by Proposition \ref{Ekaterina}, since $c_1$ is invertible. Take $m>1$ and   let $y\in\mathbb{R}^d$, so \eqref{fund_gen} holds for a finite-dimensional  invariant subspace $R(y)$ of $C(\mathbb{R}^d)$. Choosing $h\in \mathbb{R}^d$, we apply the assumption to $y-c_1^{-1}h$:
\begin{equation} \label{equno}
\sum_{i=1}^m\tau_{c_iy}(f_i)  \in W + R(y-c_1^{-1}h).
\end{equation}
Applying the operator $\tau_{h}$ to both sides of \eqref{equno} one obtains
\begin{equation} \label{eqtres}
\sum_{i=1}^m\tau_{h+c_i(y-c_1^{-1}h)}(f_i) \in \tau_h(W) + R(y-c_1^{-1}h).
\end{equation}
Subtracting \eqref{equno} from \eqref{eqtres} we get
\begin{equation}\label{inv22}
\sum_{i=2}^m(\tau_{c_iy +(I_d-c_ic_1^{-1})h}(f_i) - \tau_{c_iy}(f_i))  \in W_1+R_1(y),
\end{equation}
where $W_1=\tau_h(W)+W$ and $R_1(y)= R(y)+R(y-c_1^{-1}h)$. Clearly $W_1$ and $R_1(y)$ are finite-dimensional and $R_1(y)$ is translation invariant.

Setting $d_i = I_d-c_ic_1^{-1}$ define the functions $g_i$ by $g_i(x) = f_i(x+d_ih) - f_i(x)$. Then
  (\ref{inv22}) will get the form
\begin{equation*}
\sum_{i=2}^m\tau_{c_iy}(g_i) \in W_1+R_1(y), \text{ for all } y\in\mathbb{R}^d.
\end{equation*}
Thus, the induction step confirms us that all functions $g_i$, for $i=2,\ldots,m$, are exponential polynomials. Since $h$ is arbitrary, the proof can be finished in the same way as the proof of Theorem \ref{principalF}.
\end{proof}

\begin{corollary} \label{ku} If the functions $v_k$ are continuous, the matrices $b_i, c_i$ and $b_i^{-1}c_i-b_j^{-1}c_j$ (for $i\neq j$) are invertible, then all continuous solutions $f_i$ of \emph{(\ref{ras})} are exponential polynomials.
\end{corollary}

\section{ Distributions}

The functions which are the coefficients of the equation (\ref{LCGG})  (as well as its solutions)   could  a priori belong to a more general class than $C(\mathbb{R}^d)$. To handle  a more wide variety of situations   we will now study it in the distributional setting.

We will distinguish two variants of distributional view at the equation (\ref{LCGG}). If $u_k$ are usual (arbitrary!) functions while $v_k$ are distributions then the equation means that the sum in the left-hand side for every $y$ belongs to the linear span of the distributions $v_k$. So we come to the following setting:

\begin{theorem} \label{principal}
Assume that $\{f_k\}_{k=1}^{m}\subset \mathcal{D}(\mathbb{R}^d)'$ and, for all $y\in \mathbb{R}^d$,
\begin{equation*}
\sum_{i=1}^m\tau_{c_iy}(f_i) \in W \text{ for all } y\in\mathbb{R}^d
\end{equation*}
for an $n$-dimensional subspace $W$ of $\mathcal{D}(\mathbb{R}^d)'$. If all matrices $c_i$ and $c_i-c_j$ (for $i\neq j$) are invertible,
then all $f_k$ are continuous exponential polynomials.
\end{theorem}

\begin{proof}  The proof is similar to the proof of Theorem \ref{principalF}. The case $m=1$ follows from   the Anselone-Korevaar theorem \cite{anselone}. The induction step goes as above with the only distinction that we  apply Proposition \ref{Ekaterina} to the representation of $\mathbb{R}^d$ by shifts on the space $\mathcal{D}(\mathbb{R}^d)'$.
\end{proof}

\begin{corollary} The statements of Theorem \ref{main} and Corollary \ref{ku} extend to the case that $v_k$ and $f_i$ are locally summable.
\end{corollary}

On the other hand one can consider (\ref{LCGG}) in the case that both $u_k$ and $v_k$ are distributions. In this setting we should regard both sides of the equation as elements of $\mathcal{D}(\mathbb{R}^d\times\mathbb{R}^d )'$.

\begin{theorem} \label{principal2}
Assume that
\begin{equation} \label{delP2}
\sum_{i=1}^mf_i(x+c_iy)= \sum_{k=1}^nu_k(y)v_k(x),
\end{equation}
where $f_i, u_k, v_k\in\mathcal{D}(\mathbb{R}^d)'$, $c_i\in {GL}(d,\mathbb{R})$  and both sides of (\ref{delP2}) are considered as elements of $\mathcal{D}(\mathbb{R}^d\times\mathbb{R}^d )'$. If all matrices $c_i-c_j$ (for $i\neq j$) are invertible, then $f_k$ is a continuous exponential polynomial for $k=1,\ldots,m$.
\end{theorem}

\begin{proof}
Let us denote by $\Delta_{(h,k)}$   the general difference operator in  $\mathcal{D}(\mathbb{R}^d\times\mathbb{R}^d )'$: $$\langle\Delta_{(h,k)}F(x,y),\phi(x,y)\rangle = \langle F(x,y),\phi(x-h,y-k)-\varphi(x,y)\rangle.$$ We will use the fact that the equality
\begin{equation}\label{folfact}
\Delta_{(h,k)}(f(x+cy))=(\Delta_{h+ck}(f))(x+cy)
\end{equation}
 holds for all $f\in\mathcal{D}(\mathbb{R}^d)'$, $h,k\in\mathbb{R}^d$, and $c\in {GL}(d,\mathbb{R})$; its validity can be checked by direct calculation.

 As above we proceed by induction on $m$, the number of summands in the left-hand side of the equation \eqref{delP2}. As we have already noticed, the case $m=1$ of this equation is known (see, e.g.,  \cite{fe-56, fe} for $d=1$ and \cite{S_PhD, Sh-dist} for domains of $\mathbb{R}^d$).  Assume the result holds true whenever we have less than $m$ summands. Let $f_i,u_k,v_k$ satisfy \eqref{delP2}.

Let us apply the operator $\Delta_{(h,-c_1^{-1}h)}$ to both sides of the equation (this is equivalent to substitute $x$ by $x+h$ and $y$ by $y-c_1^{-1}h$ in the equation). Then (\ref{folfact}) implies that
\begin{eqnarray*}
\Delta_{(h,-c_1^{-1}h)}\left[ \sum_{i=1}^m f_i(x+c_iy)\right]  &=&   \sum_{i=1}^m \Delta_{(h,-c_1^{-1}h)}f_i(x+c_iy) \\
&=&  \sum_{i=1}^m (\Delta_{h-c_ic_1^{-1}h}(f_i))(x+c_iy)\\
&=&  \sum_{i=2}^m (\Delta_{(I_d-c_ic_1^{-1})h}(f_i))(x+c_iy)\\
&=&  \sum_{i=2}^m g_i(x+c_iy),
\end{eqnarray*}
with $g_i= \Delta_{(I_d-c_ic_1^{-1})h}(f_i)\in\mathcal{D}(\mathbb{R}^d)'$ for $i=2,\ldots,m$. Hence, after applying the operator $\Delta_{(h,-c_1^{-1}h)}$ to the left-hand side of the equation, we reduce by $1$ the number of summands in the equation. On the other hand, in the right-hand side of the equation we get
\[
\Delta_{(h,-c_1^{-1}h)}\left( \sum_{k=1}^nu_k(y)v_k(x)\right)  = \sum_{k=1}^n\tau_{-c_1^{-1}h}(u_k)(y)\tau_h(v_k)(x) -
\sum_{k=1}^nu_k(y)v_k(x),
\]
which is an expression of the form
\[
\sum_{k=1}^{2n}U_k(y)V_k(x)
\]
with $U_k,V_k\in\mathcal{D}(\mathbb{R}^d)'$ for $k=1,\cdots, 2n$. Hence we can use the induction hypothesis to conclude that $g_i= \Delta_{(I_d-c_ic_1^{-1})h}(f_i)\in\mathcal{D}(\mathbb{R^d})'$  is a continuous exponential polynomial for $i=2,\ldots,m$. As in the proof of Theorem \ref{principal} we conclude, using Proposition \ref{Ekaterina}, that all $f_i$  are continuous exponential polynomials.
\end{proof}

\begin{corollary} \label{corprincipal2}
Assume that
\begin{equation*} 
\sum_{i=1}^mf_i(b_ix+c_iy)= \sum_{k=1}^nu_k(y)v_k(x),
\end{equation*}
where $f_i, u_k, v_k\in\mathcal{D}(\mathbb{R}^d)'$ and $b_i,c_i\in {GL}(d,\mathbb{R})$. 
If all matrices $b_i^{-1}c_i-b_j^{-1}c_j$ (for $i\neq j$) are invertible, then $f_k$ is 
a continuous exponential polynomial for $k=1,\ldots,m$.
\end{corollary}

As a consequence, the results of Kakutani-Nagumo, Walsh, Ghurie-Olkin and others mentioned in the Introduction, extend to distributions. For example, we have

\begin{theorem} \label{GO-t}
Assume that $f_i, a_{\alpha}, b_{\beta}\in  \mathcal{D}(\mathbb{R}^d)'$ for $1\leq i\leq m$, $0\leq |\alpha|\leq r$ and $0\leq |\beta|\leq s$, and equation \eqref{G-O} is satisfied with  $A(x,y)=\sum_{|\alpha|\leq r}x^{\alpha} \cdot a_{\alpha}(y)$ and   $B(y,x)=\sum_{|\beta|\leq s}b_{\beta}(x)\cdot y^{\beta}$. Assume, furthermore, that all matrices $c_i$ (for all $i$) and $c_i-c_j$ (for $i\neq j$) are invertible. Then all $f_i$ are (in the distributional sense) ordinary polynomials.
\end{theorem}
\begin{proof}
It follows from Theorem \ref{principal} that $f_i$ are continuous exponential polynomials, which implies immediately the same about $b_{\beta}(x)$. Therefore one can apply results of \cite{Gu_O} where the statement was proved for continuous functions.
\end{proof}

\end{document}